\documentclass[11pt]{amsart}
\usepackage{amsmath}
\usepackage{amsfonts}

\usepackage{amscd}
\usepackage{amsthm}
\usepackage{amssymb} \usepackage{latexsym}
\usepackage{eufrak}
\usepackage{euscript}
\usepackage{epsfig}
\usepackage{graphics}
\usepackage{array}
\usepackage{enumerate}
\usepackage{dsfont}
\usepackage{color}
\usepackage{wasysym}

\usepackage{pdfsync}
\usepackage[colorlinks,citecolor=red,pagebackref,hypertexnames=false, breaklinks]{hyperref}

\calclayout
\allowdisplaybreaks

\newcommand{\bel}[1]{\begin{equation}\label{#1}}

\newcommand{\be}{\begin{equation}}

\newcommand{\ba}{\begin{eqnarray}}
\newcommand{\ea}{\end{eqnarray}}

\newcommand{\qe}{\end{equation}}
\newcommand{\R}{{\mathbb R}}

\newcommand{\Z}{{\mathbb Z}}

\newcommand{\dd}{\mathrm{d}}

\newcommand{\Hmm}[1]{\leavevmode{\marginpar{\tiny%
$\hbox to 0mm{\hspace*{-0.5mm}$\leftarrow$\hss}%
\vcenter{\vrule depth 0.1mm height 0.1mm width \the\marginparwidth}%
\hbox to
0mm{\hss$\rightarrow$\hspace*{-0.5mm}}$\\\relax\raggedright #1}}}

\theoremstyle{plain}
\newtheorem{thm}{Theorem}[section]
\newtheorem{lemma}[thm]{Lemma}
\newtheorem{coro}[thm]{Corollary}
\newtheorem{prop}[thm]{Proposition}

\theoremstyle{definition}
\newtheorem{defn}[thm]{Definition}
\newtheorem{rem}[thm]{Remark}

\numberwithin{equation}{section}

\begin{document}

\date{\today}
\title[Payne-Polya-Weinberger, Hile-Protter and Yang's inequalities on lattices]{Payne-Polya-Weinberger, Hile-Protter and Yang's inequalities for Dirichlet Laplace eigenvalues on integer lattices}
\author{Bobo Hua}
\email{bobohua@fudan.edu.cn}
\address{School of Mathematical Sciences, LMNS, Fudan University, Shanghai 200433, China.}

\author{Yong Lin}
\email{linyong01@ruc.edu.cn}
\address{Department of Mathematics, Information School,
Renmin University of China,
Beijing 100872, China}

\author{Yanhui Su}
\email{suyh@fzu.edu.cn}
\address{College of Mathematics and Computer Science, Fuzhou University, Fuzhou 350116, China}

\begin{abstract}
In this paper, we prove some analogues of Payne-Polya-Weinberger, Hile-Protter and Yang's inequalities for Dirichlet (discrete) Laplace eigenvalues on any subset in the integer lattice $\Z^n.$ This partially answers a question posed by Chung and Oden \cite{ChuOd00}.

 \end{abstract}
\maketitle


\section{Introduction}

The eigenvalue problem of the Laplace operator with Dirichlet boundary condition, called Dirichlet Laplacian in short, on a bounded domain $\Omega\subset\R^n$ has been extensively studied in the literature, see e.g. \cite{CouHil53,Chavel84,SY94}. We denote by
$$0<\lambda_1<\lambda_2\leq\lambda_3\leq\cdots\ \uparrow \infty$$ the spectrum of Dirichlet Laplacian on $\Omega$, counting the multiplicity of eigenvalues. In 1911, Weyl \cite{Weyl} proved that
$$\lambda_k\sim\frac{4\pi^2}{\left(\omega_n\mathrm{vol}(\Omega)\right)^{\frac{2}{n}}}k^{\frac{2}{n}},\quad k\rightarrow\infty,$$
where $\omega_n$ is the volume of the unit ball in $\R^n$ and $\mathrm{vol}(\Omega)$ is the volume of $\Omega$. Furthermore, P\'{o}lya \cite{Pol61} conjectured the eigenvalues $\lambda_k$ would satisfy
$$\lambda_k\geq\frac{4\pi^2}{(\omega_n\mathrm{vol}(\Omega))^{\frac{2}{n}}}k^{\frac{2}{n}},\quad k=1,2,3,\cdots.$$
Li and Yau \cite{LY83} proved that
$$\lambda_k\geq\frac{n}{n+2}\frac{4\pi^2}{(\omega_n\mathrm{vol}(\Omega))^{\frac{2}{n}}}k^{\frac{2}{n}},\quad
k=1,2,3,\cdots.$$

For the gaps between consecutive eigenvalues of Dirichlet Laplacian, it was first proved by Payne, Polya and Weinberger \cite{PPW} for a bounded domain in $\R^2$ and was then generalized to $\R^n$ by Thompson \cite{Thomp69} that for any $k\geq 1,$
$$\lambda_{k+1}-\lambda_k\leq\frac{4}{nk}\sum_{i=1}^k\lambda_i,$$ which is now called Payne-Polya-Weinberger inequality. Later, Hile and Protter \cite{HP80} obtained the so-called Hile-Protter inequality that
$$\sum_{i=1}^k\frac{\lambda_i}{\lambda_{k+1}-\lambda_i}\geq\frac{kn}{4}.$$

A sharp inequality was proved by Yang \cite{Yang91,CY07} that
\begin{equation}\label{eq:Yangint1}\sum_{i=1}^k(\lambda_{k+1}-\lambda_i)^2\leq\frac{4}{n}\sum_{i=1}^k\lambda_i(\lambda_{k+1}-\lambda_i),\end{equation} which implies
\begin{equation}\label{eq:Yangint2}\lambda_{k+1}\leq\left(1+\frac{4}{n}\right)\frac1k\sum_{i=1}^k\lambda_i.\end{equation}
These inequalities, \eqref{eq:Yangint1} and \eqref{eq:Yangint2}, are called Yang's first and second inequalities respectively, see \cite{AshBen96,HarStu97,Ash99,Ash02}. It is well-known, see e.g. \cite{Ash99}, that Yang's first inequality implies Yang's second inequality; the latter yields Hile-Protter inequality; Hile-Protter inequality is stronger than Payne-Polya-Weinberger inequality. We say these inequalities are universal since they apply to all bounded domains $\Omega\subset \R^n.$ Universal inequalities for eigenvalues of Dirichlet Laplacians in $\R^n$ and their generalizations to general manifolds have been studied by many authors, see e.g. \cite{Li80,YangYau80,Leung91,Har93,HarMich94,HarStu97,ChengYang05,ChengYang06,Har07,ChenCheng08,SCY08,ChengYang09,EHI09,CZL12,CZY16}

The Dirichlet Laplacian on a finite subset of a graph has been investigated in the literature of discrete analysis, see e.g. \cite{Dod84,Frie93,CouGri98,ChungYau00,BHJ14} and many others. In this paper, we consider eigenvalue problems for Dirichlet Laplacians on integer lattices. First, we recall some basic notions of Laplace operators on discrete spaces, i.e. on graphs. Let $(V,E)$ be a simple, undirected, locally finite graph with the set of vertices $V$ and the set of edges $E$ which has no isolated vertices. Two vertices $x,y$ are called neighbors, denoted by $x\sim y$, if there is an edge connecting them. The degree of a vertex $x,$ denoted by $\dd_x$ is defined as the number of neighbors of $x.$ The (discrete) Laplacian $\Delta$ on $(V,E)$ is defined as
$$\Delta f(x):=\frac{1}{d_x}\sum_{y\in V: y\sim x}f(y)-f(x),\quad \forall f:V\to\R.$$
Let $\Omega$ be a finite subset of $V.$ We define the vertex boundary of $\Omega$ as
$$\delta\Omega=\{y\in V\setminus \Omega:y\sim x\ \mbox{for some}\ x\in\Omega\}.$$ We denote by $\ell^2(\Omega,d)$ the Hilbert space of all functions defined on $\Omega$ equipped with the inner product, $$\langle f,g \rangle:=\sum_{x\in \Omega} f(x)g(x)\dd_x,\quad f,g:\Omega\to\R.$$
The Laplacian with Dirichlet boundary condition on $\Omega,$ denoted by $\Delta_{\Omega},$ is defined as, for any $f:\Omega\to \R$
$$\Delta_{\Omega} f(x):=\Delta \widetilde{f}(x),\ \ x\in \Omega,$$ where $\widetilde{f}$ is the null extension of $f$ to $V,$ i.e. $\widetilde{f}(x)=f(x),\ \forall x\in \Omega$ and $\widetilde{f}(x)=0,\ \forall x\in V\setminus\Omega.$ We will simply write $\Delta$ for $\Delta_{\Omega}$ if the subset $\Omega$ is clear in the context. One is ready to check that $\Delta_{\Omega}$ is a self-adjoint operator on $\ell^2(\Omega,d),$ see e.g. \cite{BHJ14}.  Moreover, $\lambda$ is an eigenvalue of $-\Delta_{\Omega}$ if and only if there is a function $u:\Omega\cup\delta\Omega\to\R$ such that
\begin{eqnarray}
\left\{\begin{array}{ll}
-\Delta u=\lambda u,&\mbox{in}\ \Omega,\\
u=0,&\mbox{on}\ \delta\Omega.
\end{array}\right.
\end{eqnarray}
 For a finite subset $\Omega,$ we write the spectrum of the Dirichlet Laplacian on $\Omega,$ i.e. eigenvalues of $-\Delta_{\Omega},$ as
$$0<\lambda_1\leq\lambda_2\leq\cdots\leq \lambda_N\leq 2$$ where $N:=\sharp\Omega$ denotes the number of vertices in $\Omega.$ Note that by the definition of the Laplacian all the eigenvalues are bounded above by $2.$


The $n$-dimensional integer lattice graph, denoted by $\Z^n$, is of particular interest which serves as the discrete counterpart of $\R^n,$ see Section~\ref{s:pre} for the definition. Since $\Z^n$ is bipartite, the spectrum of the Dirichlet Laplacian on any finite subset of $\Z^n$ is symmetric with respect to $1$, i.e. \begin{equation}\label{eq:bip}\lambda_k=2-\lambda_{N+1-k},\quad \forall 1\leq k\leq N,\end{equation} see \cite[Lemma~3.4]{BHJ14}. In particular, this implies that for any $k\leq N,$ $$\sum_{i=1}^k(1-\lambda_i)\geq 0,$$ see Proposition~\ref{prop:bipartite}.

Chung and Oden \cite[pp. 268]{ChuOd00} proposed a question whether one can generalize Payne-Polya-Weinberger inequality to the discrete setting, i.e. for eigenvalues of Dirichlet Laplacian on subsets in $\Z^n.$  In this paper, following the proof strategies in the continuous setting, we prove a discrete analogue of Payne-Polya-Weinberger inequality in $\Z^n,$ which partially answers the above question.


\begin{thm}[Payne-Polya-Weinberger type inequality]\label{thm:PPW} Let $\Omega$ be a finite subset of $\Z^n$ and $\lambda_i$
be the $i$-th eigenvalue of the Dirichlet Laplacian on $\Omega.$ Then for any $1\leq k\leq \sharp\Omega-1,$
\begin{equation}\label{thm:PPWin}\tag{Payne-Polya-Weinberger}\lambda_{k+1}-\lambda_{k}\leq \frac{4}{n}\frac{\sum_{i=1}^k\lambda_i}{\sum_{i=1}^k(1-\lambda_i)}.\end{equation} 

\end{thm}
As a corollary, by setting $k=1$, we obtain the upper bound estimate for the first gap of Dirichlet eigenvalues of a subset in $\Z^n,$ see Section~\ref{s:app} for the proof.
\begin{coro}\label{coro:first gap} Let $\Omega$ be a finite subset of $\Z^n$ with $\sharp\Omega\geq 2$ and $\lambda_i$
be the $i$-th eigenvalue of the Dirichlet Laplacian on $\Omega.$ Then
$$\lambda_2-\lambda_1\leq \frac{4\lambda_1}{n(1-\lambda_1)},$$ where the right hand side is to be interpreted as infinity if $\lambda_1=1.$ Moreover, $$\lambda_2\leq 9\lambda_1.$$ 
\end{coro}

We also obtain a discrete analogue of Hile-Protter inequality, following e.g. \cite{Ash99}.
\begin{thm}[Hile-Protter type inequality]\label{thm:HP} Let $\Omega$ be a finite subset of $\Z^n$ and $\lambda_i$
be the $i$-th eigenvalue of the Dirichlet Laplacian on $\Omega.$ Then for any $1\leq k\leq \sharp\Omega-1,$
\begin{equation}\label{eq:HPthm}\tag{Hile-Protter}\sum_{i=1}^k\frac{\lambda_i}{\lambda_{k+1}-\lambda_i}\geq \frac{n}{4}\sum_{i=1}^k(1-\lambda_i),\end{equation} where the left hand side is to be interpreted as infinity if $\lambda_{k+1}=\lambda_k.$ 
\end{thm}

Following the arguments in \cite{Yang91,Ash99,CY07}, we are able to generalize Yang's first inequality to the discrete setting.

\begin{thm}[Yang type first inequality]\label{thm:Yang} Let $\Omega$ be a finite subset of $\Z^n$ and $\lambda_i$
be the $i$-th eigenvalue of the Dirichlet Laplacian on $\Omega.$ Then for any $1\leq k\leq \sharp\Omega-1,$
\begin{equation}\label{eq:Yang}\tag{Yang-1}\sum_{i=1}^k(\lambda_{k+1}-\lambda_{i})^2(1-\lambda_i)\leq \frac{4}{n}\sum_{i=1}^k(\lambda_{k+1}-\lambda_i)\lambda_i.\end{equation}\end{thm}

By Chebyshev's inequality, Yang type second inequality on $\Z^n$ follows from the first one, see \cite{CZY16}.
\begin{thm}[Yang type second inequality]\label{thm:Yang2}
Let $\Omega$ be a finite subset of $\Z^n$ and $\lambda_i$
be the $i$-th eigenvalue of the Dirichlet Laplacian on $\Omega.$ If $\lambda_{k+1}\leq 1+\frac4n$ for some $1\leq k\leq \sharp \Omega-1,$ then
\begin{equation}\label{eq:Yang2}\tag{Yang-2}\lambda_{k+1}\leq \frac{(1+\frac4n)\sum_{i=1}^k\lambda_i-\sum_{i=1}^k\lambda_i^2}{\sum_{i=1}^k(1-\lambda_i)}.\end{equation}
\end{thm}
Note that the additional condition $\lambda_{k+1}\leq 1+\frac4n$ is trivial for $n\leq 4.$ 

Using the argument in \cite{CY07}, Yang type first inequality yields the following corollary, see Section~\ref{s:app} for the proof.
\begin{coro}\label{coro:ratio}Let $\Omega$ be a finite subset of $\Z^n$ and $\lambda_i$
be the $i$-th eigenvalue of the Dirichlet Laplacian on $\Omega.$ If $\lambda_k<1$ for some $1\leq k\leq \sharp \Omega-1,$ then
$$\lambda_{k+1}\leq \left(1+\frac{4}{n(1-\lambda_k)}\right)k^{\frac{2}{n(1-\lambda_k)}}\lambda_1.$$ In particular, if $\lambda_k\leq 1-\delta$ for some $\delta>0,$ then
\begin{equation}\label{eq:est1}\lambda_{k+1}\leq \left(1+\frac{4}{n\delta}\right)k^{\frac{2}{n\delta}}\lambda_1.\end{equation}
\end{coro} The estimate \eqref{eq:est1} is analogous to the estimate in \cite{CY07} for bounded domains in $\R^n$, i.e.
$$\lambda_{k+1}\leq \left(1+\frac4n\right)k^{\frac{2}{n}}\lambda_1.$$ Although the number of eigenvalues is finite for a fixed subset $\Omega,$ the estimate \eqref{eq:est1} is still interesting since it holds for any subset in $\Z^n$ for which the first eigenvalue $\lambda_1$ could be arbitrarily small.


Similar to the continuous setting, we obtain the relations between these universal inequalities. Under the condition that $\lambda_{k+1}\leq 1+\frac4n,$
$$\eqref{eq:Yang} \Longrightarrow\ \eqref{eq:Yang2} \Longrightarrow\ \eqref{eq:HPthm} \Longrightarrow\ \eqref{thm:PPWin}.$$
The first implication follows from the proof of Theorem~\ref{thm:Yang2}. The second one follows from a convexity argument in \cite{Ash99}, see Proposition~\ref{prop:Yang2toHP}. It is easy to see that the last implication holds, see the proof of Theorem~\ref{thm:PPW}.

The paper is organized as follows: In next section, we introduce some basic properties related to the analysis on graphs. Section~\ref{s:proof} is devoted to the proofs of main results, Theorem~\ref{thm:PPW}, \ref{thm:HP}, \ref{thm:Yang} and  \ref{thm:Yang2}. In the last section, some applications of universal inequalities are given.


\section{Preliminaries}\label{s:pre}
Let $G=(V,E)$ be a simple, undirected, locally finite graph. For the convenience, we introduce the following notion,
$$\mu:V\times V\to \{0,1\},$$ $$\quad(x,y)\mapsto \mu_{xy},$$ such that $\mu_{xy}=1$ if and only if $x\sim y.$ So that the degree of a vertex $x$ is given by $\dd_x:=\sum_{y\in V}\mu_{xy}.$ We write for simplicity $$\sum_{x} f(x):=\sum_{x\in V}f(x)$$ if it is clear in the context.

We denote by $\Z^n:=\{x=(x_1,x_2,\cdots,x_n):x_i\in\Z, 1\leq i\leq n\}$ the set of integer $n$-tuples $\R^n.$ The $n$-dimensional integer lattice graph, still denoted by $\Z^n,$ is the graph consisting of the set of vertices $V=\Z^n$ and the set of edges $E=\{\{x,y\}: x,y\in \Z^n, \sum_{i=1}^n|x_i-y_i|=1\}.$ Note that, $\dd_x=2n$ for any $x\in \Z^n.$



Given $f: V\to\R$ and $x,y\in V,$ we denote by $\nabla_{xy}f:= f(y)-f(x)$ the difference of the function $f$ on the vertices $x$ and $y.$ One can easily check that for any function $f,g:V\to \R$ and any $x,y\in V,$
\begin{equation}\label{eq:ele1} f(x)g(x)+f(y)g(y)=\frac{1}{2}[(f(x)+f(y))(g(x)+g(y))+\nabla_{xy}f\nabla_{xy}g].
\end{equation} We introduce the following discrete analogue to the square of the gradient of a function.
\begin{defn}\label{d:carre du}
  The gradient form $\Gamma,$ called the ``carr\'e du champ" operator, is defined by, for $f,g:V\to\R$ and $x\in V$,
  \begin{eqnarray}\Gamma(f,g)(x)&=&\frac12(\Delta(fg)-f\Delta g-g\Delta f)(x)\label{def:gamma}\\&=&\frac{1}{2\dd_x}
  \sum_{y}\mu_{xy}\nabla_{xy}f\nabla_{xy}g.\nonumber\end{eqnarray} We write $\Gamma(f):=\Gamma(f,f)$ for simplicity.
\end{defn}

The following lemma is useful.
\begin{lemma}Let $u$ be a function on $\Z^n$ of finite support and $\{x_\alpha\}_{\alpha=1}^n$ be standard coordinate functions. Then
\begin{equation}\label{eq:grad1} \sum_{\alpha=1}^n\frac12\sum_{x,y}|\nabla_{xy} (x_\alpha)|^2|\nabla_{xy} u|^2\mu_{xy}=\sum_{x}\Gamma(u)(x)\dd_x,
\end{equation} and
\begin{equation}\label{eq:grad2}\sum_{\alpha=1}^n\sum_{x}\Gamma(x_\alpha,u)^2(x)\dd_x\leq \frac{1}{2n}\sum_{x}\Gamma(u)(x)\dd_x.
\end{equation}
\end{lemma}
\begin{proof} The first assertion follows from the fact that for any $x,y\in V$ satisfying $x\sim y,$
$$\sum_{\alpha=1}^n |\nabla_{xy} (x_\alpha)|^2=1.$$ For the second assertion, we denote by $\{e_\alpha\}_{\alpha=1}^n$ the set of natural orthonormal bases of $\R^n$ where $e_\alpha$ is the unit vector whose $\alpha$-th coordinate is $1.$ Then
\begin{eqnarray*}\sum_{\alpha=1}^n\sum_{x}\Gamma(x_\alpha,u)^2(x)\dd_x&=&\sum_{\alpha=1}^n\sum_{x}\left(\frac{1}{2\dd_x}(u(x+e_{\alpha})-u(x)+u(x)-u(x-e_\alpha))\right)^2\dd_x\\
&\leq&\frac{1}{4n}\sum_{\alpha=1}^n\sum_{x}((u(x+e_{\alpha})-u(x))^2+(u(x)-u(x-e_\alpha))^2)\\
&=&\frac{1}{2n}\sum_{x}\Gamma(u)(x)\dd_x.
\end{eqnarray*} This proves the lemma.
\end{proof}


The following Green's formula is well-known, see e.g. \cite{Grigoryan09}.
\begin{prop}\label{prop:Green} Let $(V,E)$ be a graph, $f$ be a function with finite support on $V$ and $g$ be any function on $V.$ Then $$\sum_{x\in V}(\Delta f)(x) g(x)\dd_x=-\frac{1}{2}\sum_{x,y\in V}\mu_{xy}\nabla_{xy}f\nabla_{xy}g=\sum_{x\in V}(\Delta g)(x) f(x)\dd_x.$$
\end{prop}

The following Chebyshev's inequality is also well-known.
\begin{prop}\label{prop:Che} Let $\{a_i\}_{i=1}^N$ and $\{b_i\}_{i=1}^N$ be sequences of real numbers satisfying $$a_1\leq a_2\leq \cdots\leq a_N,\quad b_1\leq b_2\leq \cdots\leq b_N.$$ Then
$$\frac1N\sum_{i=1}^N a_ib_i\geq\left( \frac1N\sum_{i=1}^N a_i \right)\left(\frac1N\sum_{i=1}^N b_i\right).$$
\end{prop}

Since $\Z^n$ is bipartite, the spectrum of Dirichlet Laplace on any subset $\Omega$ is symmetric, see \eqref{eq:bip}. This yields the following proposition.
\begin{prop}\label{prop:bipartite} Let $\Omega$ be a finite subset of $\Z^n$ and $\lambda_i$
be the $i$-th eigenvalue of the Dirichlet Laplacian on $\Omega.$ Then for any $k\leq \sharp \Omega,$ $$\sum_{i=1}^k(1-\lambda_i)\geq 0.$$ The equality holds if and only if $k=\sharp \Omega.$
\end{prop}
\begin{proof} We define
$\mathcal{K}_{1+}:=\{i:1\leq i\leq k, \lambda_i>1\},\mathcal{K}_{1-}:=\{i:1\leq i\leq k, \lambda_i<1\},$ and $\mathcal{K}_{1}:=\{i:1\leq i\leq k, \lambda_i=1\}$ respectively. By \eqref{eq:bip}, we know that $\sharp\mathcal{K}_{1+}\leq \sharp\mathcal{K}_{1-}$ and $$\widehat{\mathcal{K}_{1+}}:=\{N+1-i: i\in \mathcal{K}_{1+}\}\subset\mathcal{K}_{1-}.$$ This implies that
\begin{eqnarray*}\sum_{i=1}^k(1-\lambda_i)&=&\left(\sum_{i\in \mathcal{K}_{1+}}+\sum_{i\in \mathcal{K}_{1-}}+\sum_{i\in \mathcal{K}_{1}}\right)(1-\lambda_i)\\
&=&\left(\sum_{i\in \mathcal{K}_{1+}}+\sum_{i\in \widehat{\mathcal{K}_{1+}}}+\sum_{i\in \mathcal{K}_{1-}\setminus\widehat{\mathcal{K}_{1+}}}\right)(1-\lambda_i)\\
&=&\sum_{i\in \mathcal{K}_{1-}\setminus\widehat{\mathcal{K}_{1+}}}(1-\lambda_i)\geq 0.\end{eqnarray*} The equality case is easy to verify in the above argument.
\end{proof}

The Rayleigh quotient characterization for the first eigenvalue of Dirichlet Laplacian on a finite subset $\Omega\subset V$  reads as
$$\lambda_1=\inf_{f:\Omega\to\R,f\not\equiv 0}\frac{\sum_{x\in V}\Gamma(\widetilde{f})(x)\dd_x}{\sum_{x\in V}\widetilde{f}^2(x)\dd_x},$$ where $\widetilde{f}$ is the null extension of $f$ to $V.$ We say that $\Omega$ is connected if the induced subgraph on $\Omega$ is connected.
\begin{prop}\label{prop:firsteigen}Let $\Omega$ be a finite connected subset of $\Z^n$ with $\sharp\Omega\geq 2$ and $\lambda_1$
be the first eigenvalue of the Dirichlet Laplacian on $\Omega.$ Then
$$\lambda_1\leq 1-\frac{1}{2n}.$$
\end{prop}
\begin{proof} By the assumption, there are two vertices $v_1,v_2$ in $\Omega$ such that $v_1\sim v_2.$ Set $H=\{v_1,v_2\}.$ We consider the eigenvalues of Dirichlet Laplacian on $H,$ and denote by $\lambda_1^H$ its first eigenvalue. By Rayleigh quotient characterization, one is ready to see that $\lambda_1\leq\lambda_1^H.$ The proposition follows from $\lambda_1^H=1-\frac{1}{2n}.$
\end{proof}

By the convexity argument of \cite{Ash99}, we have the following implication, $$\eqref{eq:Yang2} \Longrightarrow\ \eqref{eq:HPthm}.$$
\begin{prop}\label{prop:Yang2toHP} Let $\Omega$ be a finite subset of $\Z^n$ and $\lambda_i$
be the $i$-th eigenvalue of the Dirichlet Laplacian on $\Omega.$ Suppose that \eqref{eq:Yang2} holds, then we have \eqref{eq:HPthm}.
\end{prop}
\begin{proof} Without loss of generality, we may assume that $\lambda_k<\lambda_{k+1}.$ Set $g(x):=\frac{x}{\lambda_{k+1}-x}.$ It is easy to see that the function $g$ is convex in $x\in(-\infty,\lambda_{k+1}).$ Hence
$$\frac1k\sum_{i=1}^k\frac{\lambda_i}{\lambda_{k+1}-\lambda_i}=\frac 1k\sum_{i}g(\lambda_i)\geq g\left(\frac1k\sum_i{\lambda_i}\right)=\frac{\frac1k\sum_i\lambda_i}{\lambda_{k+1}-\frac1k\sum_i\lambda_i},$$ where we have used Jensen's inequality for convex function $g(\cdot).$ By plugging \eqref{eq:Yang2} into the above inequality and using $$\frac{1}{k}\sum_i\lambda_i^2\geq \left(\frac1k\sum_{i=1}^k\lambda_i\right)^2,$$ we prove the proposition.
\end{proof}

\section{Proof of main results}\label{s:proof}
In this section, we prove the main results, Theorem~\ref{thm:PPW}, \ref{thm:HP} and \ref{thm:Yang}, following the arguments in \cite{Ash99,CY07}.

Let $\Omega$ be a finite subset of $\Z^n$ and $\lambda_k$ be the $k$-th eigenvalue of the Dirichlet Laplacian on $\Omega.$ Let $k\leq \sharp \Omega-1.$ For any $1\leq i\leq k,$ set $u_i$ be the normalized eigenvectors associated with the eigenvalue $\lambda_i,$ i.e. $$-\Delta u_i(x)=\lambda_i u_i(x), \quad \sum_{x\in \Omega} u_i(x)u_j(x)\dd_x=\delta_{ij}$$ for any $1\leq i,j\leq k.$ For convenience, we extend the eigenvectors to the whole graph $\Z^n,$ still denoted by $\{u_i\}_{i=1}^k,$ such that $u_i(y)=0,\ \forall y\in \Z^n\setminus \Omega.$  It is easy to check that $$\sum_{x\in \Z^n}\Gamma(u_i)(x)\dd_x=\lambda_i.$$
 Let $g$ be one of the coordinate functions, i.e. $g=x_\alpha$ for some $1\leq\alpha\leq n.$ It is easy to check that
$$\Delta g=0, \quad \Delta (g^2)=\frac{1}{n}.$$

We define for any $1\leq i\leq k,$
$$\varphi_i=gu_i-\sum_{j=1}^k a_{ij}u_j,$$ where $a_{ij}=\sum_{x\in \Z^n} g(x)u_i(x)u_j(x)\dd_x.$ This yields \begin{equation}\label{eq:orthogonal}\sum_{x\in \Z^n}\varphi_i(x) u_j(x)\dd_x=0,\quad \forall\ 1\leq j\leq k.\end{equation} Note that $a_{ij}=a_{ji}$ for any $1\leq i,j\leq k.$ Set $$b_{ij}:=\sum_{x} u_j(x)\Gamma(g,u_i)(x)\dd_x.$$

We have the following proposition.
\begin{prop} $$2b_{ij}=(\lambda_i-\lambda_j)a_{ij},\quad \forall\ 1\leq i,j\leq k.$$
\end{prop}
\begin{rem} This implies that $b_{ij}$ is anti-symmetric, i.e. $b_{ij}=-b_{ji}.$
\end{rem}
\begin{proof} By Green's formula, Proposition~\ref{prop:Green}, and \eqref{def:gamma}, \begin{eqnarray*} \lambda_j a_{ij}&=&\sum_{x}(-\Delta u_{j}(x))u_i(x)g(x)\dd_x=-\sum_{x} u_{j}(x)\Delta (u_ig)(x)\dd_x\\
&=&-\sum_{x} u_{j}(x)\left((\Delta u_i)(x) g(x)+ (\Delta g)(x) u_i(x)+2\Gamma(g,u_i)(x)\right)\dd_x\\
&=&\lambda_i a_{ij}-2b_{ij}.
\end{eqnarray*}
\end{proof}

Note that
\begin{eqnarray*} \Delta \varphi_i&=&\Delta (gu_i)-\sum_{j}a_{ij}\Delta u_j=(\Delta g) u_i+g\Delta u_i+2\Gamma(g,u_i)+\sum_ja_{ij}\lambda_j u_j\\
&=&-\lambda_i g u_i+2\Gamma(g,u_i)+\sum_{j} a_{ij}\lambda_j u_j.
\end{eqnarray*} Multiplying $-\varphi_i$ on both sides of the above equation and summing over $x\in \Z^n$ with weights $\dd_x,$ we get
$$\sum_{x}\Gamma(\varphi_i)(x)\dd_x=\lambda_i\sum_{x} g(x)u_i(x)\varphi_i(x)\dd_x+K_g(u_i),$$ where \begin{equation}\label{eq:Kg1}K_g(u_i):=-2\sum_{x}\Gamma(g,u_i)(x)\varphi_i(x)\dd_x\end{equation} and we have used Green's formula, Proposition~\ref{prop:Green}, and \eqref{eq:orthogonal}. For the first term on the right hand side of the above equation, by using \eqref{eq:orthogonal}
$$\lambda_i\sum_{x} g(x)u_i(x)\varphi_i(x)\dd_x=\lambda_i\sum_{x} (g(x)u_i(x)-\sum_j a_{ij}u_j(x))\varphi_i(x)\dd_x=\lambda_i\sum_{x}\varphi_i^2(x)\dd_x.$$ Hence
$$\sum_{x}\Gamma(\varphi_i)(x)\dd_x=\lambda_i\sum_{x}\varphi_i^2(x)\dd_x+K_g(u_i),$$
Moreover, by \eqref{eq:orthogonal} and the Rayleigh quotient characterization of $\lambda_{k+1},$
$$\lambda_{k+1}\sum_{x}\varphi_i^2(x)\dd_x\leq \sum_{x}\Gamma(\varphi_i)(x)\dd_x.$$
This yields that
\begin{equation}\label{eq:lam1}0\leq (\lambda_{k+1}-\lambda_i)\sum_{x}\varphi_i^2(x)\dd_x\leq K_g(u_i).\end{equation}

\begin{eqnarray*}K_g(u_i)&=&-2\sum_{x}\left(g(x)u_i(x)-\sum_ja_{ij}u_j(x)\right)\Gamma(g,u_i)(x)\dd_x\\
&=&-2\sum_{x}g(x)u_i(x)\Gamma(g,u_i)(x)\dd_x+2\sum_{j}a_{ij}b_{ij}\\
&=:& -I+\sum_{j}(\lambda_i-\lambda_j)a_{ij}^2.\end{eqnarray*}
Here by the symmetrization of $x,y,$ the elementary equality \eqref{eq:ele1} and Green's formula, Proposition~\ref{prop:Green}, we have \begin{eqnarray*} I&=& \sum_{x,y}\mu_{xy}g(x)u_i(x)\nabla_{xy}g\nabla_{xy}u_i=\frac{1}{2}\sum_{x,y}\mu_{xy}(g(x)u_i(x)+g(y)u_i(y))\nabla_{xy}g\nabla_{xy}u_i\\
&=&\frac14\sum_{x,y}\mu_{xy}\nabla_{xy}(g^2)\nabla_{xy}(u_i^2)+I_g(u_i)\\
&=&-\frac{1}{2}\sum_x\Delta(g^2)(x) u_i^2(x)\dd_x+I_g(u_i)\\&=&-\frac{1}{2n}+I_g(u_i),
\end{eqnarray*} where $$I_g(u_i):=\frac14\sum_{x,y}\mu_{xy}|\nabla_{xy}g|^2|\nabla_{xy}u_i|^2.$$
Hence \begin{equation}\label{eq:lam2}K_g(u_i)=\frac{1}{2n}-I_g(u_i)+\sum_{j}(\lambda_i-\lambda_j)a_{ij}^2.\end{equation}



Now we are ready to prove Theorem~\ref{thm:HP}.
\begin{proof}[Proof of Theorem~\ref{thm:HP}] Without loss of generality, we may assume $\lambda_k<\lambda_{k+1}.$
We claim that for any $1\leq i\leq k,$
\begin{equation}\label{eq:pf:1} K_g(u_i)\leq \frac{4}{\lambda_{k+1}-\lambda_i}\sum_{x}\Gamma(g,u_i)^2(x) \dd_x.
\end{equation}
To prove the claim, it suffices to assume $K_g(u_i)>0.$ By \eqref{eq:Kg1},
$$0<K_g(u_i)^2\leq 4\left(\sum_x\Gamma(g,u_i)^2(x)\dd_x\right)\left(\sum_x\varphi_i^2(x)\dd_x\right),$$ which implies that $\sum_x\varphi_i^2(x)\dd_x>0.$ Hence, combining this with \eqref{eq:lam1}, we have for any $1\leq i\leq k,$
$$\lambda_{k+1}-\lambda_i\leq \frac{K_g(u_i)}{\sum_x\varphi_i^2(x)\dd_x}\leq \frac{\sum_x\Gamma(g,u_i)^2(x)\dd_x}{K_g(u_i)}.$$ By using \eqref{eq:lam2},
$$\frac{1}{2n}-I_g(u_i)+\sum_{j}(\lambda_i-\lambda_j)a_{ij}^2=K_g(u_i)\leq \frac{4}{\lambda_{k+1}-\lambda_i}\sum_x\Gamma(g,u_i)^2(x)\dd_x.$$

Summing over $i$ from $1$ to $k$ and noting that $a_{ij}$ is symmetric, we have
$$\frac{k}{2n}-\sum_iI_g(u_i)\leq 4\sum_i\frac{\sum_x\Gamma(g,u_i)^2(x)\dd_x}{\lambda_{k+1}-\lambda_i}.$$ 
By choosing $g=x_\alpha$ for $1\leq \alpha\leq n$ in the above inequality and summing over $\alpha,$ noting that \eqref{eq:grad1} and \eqref{eq:grad2}, we have
$$\frac{k}{2}-\frac{1}{2}\sum_{i}\lambda_i\leq \frac{2}{n}\sum_{i=1}^k\frac{\lambda_i}{\lambda_{k+1}-\lambda_i}.$$ This proves the theorem.

\end{proof}

In particular, Hile-Protter type inequality implies Payne-Polya-Weinberger type inequality.
\begin{proof}[Proof of Theorem~\ref{thm:PPW}] This follows from \eqref{eq:HPthm} by using $\lambda_{k+1}-\lambda_i\geq \lambda_{k+1}-\lambda_k,$ for any $1\leq i\leq k.$
\end{proof}

Now we prove an analogue to Yang's first inequality on Dirichlet eigenvalues in $\Z^n$.
\begin{proof}[Proof of Theorem~\ref{thm:Yang}] Multiplying $(\lambda_{k+1}-\lambda_i)^2$ on both sides of \eqref{eq:lam2} and summing over $i$ from $1$ to $k,$ we get
\begin{eqnarray}\label{eq:eq7}F&:=&\sum_{i}(\lambda_{k+1}-\lambda_{i})^2K_g(u_i)\nonumber\\&=&\sum_{i}(\lambda_{k+1}-\lambda_i)^2(\frac{1}{2n}-I_g(u_i))+\sum_{i,j}(\lambda_{i}-\lambda_j)(\lambda_{k+1}-\lambda_{i})^2a_{ij}^2\nonumber\\
&=&\sum_{i}(\lambda_{k+1}-\lambda_i)^2(\frac{1}{2n}-I_g(u_i))-4\sum_{i,j}(\lambda_{k+1}-\lambda_i)b_{ij}^2,\end{eqnarray} where we have used the anti-symmetry of $b_{ij}.$
Multiplying $(\lambda_{k+1}-\lambda_i)^2$ on both sides of \eqref{eq:lam1} and summing over $i$ from $1$ to $k,$ we have
\begin{equation}\label{eq:eq5} \sum_{i}(\lambda_{k+1}-\lambda_i)^3 \sum_{x}\varphi_i^2(x)\dd_x\leq F.
\end{equation}

By \eqref{eq:orthogonal}, for any $d_{ij}\in\R,$ $1\leq i,j\leq k,$ we have
\begin{eqnarray*} &&F^2=\left(-2\sum_i(\lambda_{k+1}-\lambda_i)^2\sum_x\varphi_i(x)\Gamma(g,u_i)(x)\dd_x\right)^2\\&\leq &4 \left(\sum_i\sum_x (\lambda_{k+1}-\lambda_i)^3\varphi_i^2(x)\dd_x\right)\left(\sum_i\sum_x\left[(\lambda_{k+1}-\lambda_i)^\frac12\Gamma(g,u_i)(x)-\sum_{j=1}^kd_{ij}u_j(x)\right]^2\dd_x\right)\\ &\leq& 4F\sum_i\sum_x\left[(\lambda_{k+1}-\lambda_i)\Gamma(g,u_i)(x)^2-2\sum_jd_{ij}(\lambda_{k+1}-\lambda_i)^\frac12u_j(x)\Gamma(g,u_i)(x)\right.\\
&&\quad\quad\quad\left.+\left(\sum_{j=1}^kd_{ij}u_j(x)\right)^2\right]\dd_x.
\end{eqnarray*} In fact, we use the orthogonality condition \eqref{eq:orthogonal} to introduce an additional term involving $d_{ij}$ and then apply the Cauchy-Schwarz inequality in the second line above.
This yields that
$$F\leq 4\sum_i\sum_x(\lambda_{k+1}-\lambda_i)\Gamma(g,u_i)(x)^2\dd_x+4\left[-2\sum_{i,j}d_{ij}(\lambda_{k+1}-\lambda_i)^\frac12b_{ij}+\sum_{i,j}d_{ij}^2\right].
$$ By setting $d_{ij}=(\lambda_{k+1}-\lambda_i)^\frac12b_{ij},$ we get
$$F\leq 4\sum_i\sum_x(\lambda_{k+1}-\lambda_i)\Gamma(g,u_i)(x)^2\dd_x-4\sum_{i,j}(\lambda_{k+1}-\lambda_i)b_{ij}^2.$$ By \eqref{eq:eq7},
$$\sum_{i}(\lambda_{k+1}-\lambda_i)^2(\frac{1}{2n}-I_g(u_i))\leq 4\sum_i\sum_x(\lambda_{k+1}-\lambda_i)\Gamma(g,u_i)(x)^2\dd_x.$$ By plugging $g=x_\alpha,$ $1\leq \alpha\leq n,$ into the above inequality and summing over $\alpha,$ noting that \eqref{eq:grad1} and \eqref{eq:grad2}, we obtain
$$\sum_{i}(\lambda_{k+1}-\lambda_i)^2(\frac{1}{2}-\frac{\lambda_i}{2})\leq \frac{2}{n}\sum_i(\lambda_{k+1}-\lambda_i)\lambda_i.$$ This proves the theorem.

\end{proof}

Now we are ready to prove Yang type second inequality.
\begin{proof}[Proof of Theorem~\ref{thm:Yang2}]Without loss of generality, we may assume that $\lambda_{k+1}>\lambda_1,$ otherwise $\lambda_1=\lambda_2=\cdots=\lambda_{k+1}$ which implies \eqref{eq:Yang2} trivially. By Yang type first inequality, \eqref{eq:Yang},
$$\frac1k\sum_i(\lambda_{k+1}-\lambda_i)\left[(\lambda_{k+1}-\lambda_i)(1-\lambda_i)-\frac4n\lambda_i\right]\leq 0.$$ Set $a_i:=\lambda_{k+1}-\lambda_i$ and $b_i:=(\lambda_{k+1}-\lambda_i)(1-\lambda_i)-\frac4n\lambda_i.$ The function $$f(x):=(\lambda_{k+1}-x)(1-x)-\frac4nx$$ is non-increasing in $(-\infty,\frac12(1+\frac4n+\lambda_{k+1}))$ which implies that $b_i$ is non-increasing. Using Chebyshev's inequality, see Proposition~\ref{prop:Che}, we have
$$\left(\lambda_{k+1}-\frac1k\sum_{i=1}^k\lambda_i\right)\left[\lambda_{k+1}-(1+\frac4n+\lambda_{k+1})\frac1k\sum_{i=1}^k\lambda_i+\frac1k\sum_{i=1}^k\lambda_i^2\right]\leq 0.$$ Note that by $\lambda_{k+1}>\lambda_1,$ $$\lambda_{k+1}>\frac1k\sum_{i=1}^k\lambda_i.$$ Then
$$\lambda_{k+1}\leq \frac{(1+\frac4n)\frac1k\sum_{i=1}^k\lambda_i-\frac1k\sum_{i=1}^k\lambda_i^2}{1-\frac1k\sum_{i=1}^k\lambda_i},$$ which proves the theorem.
\end{proof}

\section{Applications}\label{s:app}
In this section, we collect some applications of universal inequalities obtained before.

We prove the first gap estimate of Dirichlet Laplacian by universal inequalities.
\begin{proof}[Proof of Corollary~\ref{coro:first gap}] The first assertion follows from Theorem~\ref{thm:PPW} by setting $k=1.$

For the second assertion, without loss of generality, we may assume $\lambda_1<1.$ Since the spectrum of Dirichlet Laplacian on $\Omega$ is the union of the spectra of Dirichlet Laplacian on its connected components. Suppose that $\Omega'$ is one of connected components of $\Omega$ whose first Dirichlet eigenvalue is $\lambda_1.$ It is easy to see that $\sharp\Omega'\geq 2.$ By Proposition~\ref{prop:firsteigen}, $$\lambda_1\leq 1-\frac{1}{2n}.$$ Combining this with the first assertion, we prove the second one.
\end{proof}

By Yang type second inequality, Theorem~\ref{thm:Yang2}, we have the following corollary.
\begin{coro}Let $\Omega$ be a finite subset of $\Z^n$ and $\lambda_i$
be the $i$-th eigenvalue of the Dirichlet Laplacian on $\Omega.$ If $\lambda_{k+1}\leq 1+\frac4n$ for some $k\leq \sharp \Omega,$ then
$$\frac1k\sum_{i=1}^k (\lambda_i-\overline{\lambda})^2\leq \frac{4}{n}\overline{\lambda},$$ where $\overline{\lambda}:=\frac1k\sum_{i=1}^k \lambda_i.$
\end{coro}
\begin{proof}
By Yang type second inequality, \eqref{eq:Yang2}, and the fact $\overline{\lambda}\leq \lambda_{k+1},$ $$\overline{\lambda}\leq \frac{(1+\frac4n)\overline{\lambda}-\frac1k\sum_{i=1}^k\lambda_i^2}{1-\overline{\lambda}}.$$ The corollary follows.

\end{proof}

In \cite{CY07}, Cheng and Yang studied the bound of $\lambda_{k+1}/\lambda_1$ for a bounded domain $\Omega\subset\R^n$ and proved that $$\lambda_{k+1}\leq \left(1+\frac4n\right)k^{\frac{2}{n}}\lambda_1.$$ Closely following their argument, we obtain its discrete analogue, see Corollary~\ref{coro:ratio}.

\begin{thm}\label{recursion}
Let $\lambda_1\leq\lambda_2\leq\cdots\leq\lambda_{k+1}$ be any positive numbers and $B>0$ satisfying
\begin{eqnarray}\label{eq:recursion}
\sum_{i=1}^k(\lambda_{k+1}-\lambda_i)^2\leq\frac{4B}{n}\sum_{i=1}^k\lambda_i(\lambda_{k+1}-\lambda_i).
\end{eqnarray}
Define
$$\Lambda_k=\frac{1}{k}\sum_{i=1}^k\lambda_i,\quad T_k=\frac{1}{k}\sum_{i=1}^k\lambda_i^2,
\quad F_k=\left(1+\frac{2B}{n}\right)\Lambda_k^2-T_k.$$
Then we have
\begin{eqnarray*}
F_{k+1}\leq C(n,k,B)\left(\frac{k+1}{k}\right)^{\frac{4B}{n}}F_k,
\end{eqnarray*}
where
$$C(n,k,B)=1-\frac{B}{3n}\left(\frac{k}{k+1}\right)^{\frac{4B}{n}}\frac{\left(1+\frac{2B}{n}\right)
\left(1+\frac{4B}{n}\right)}{(k+1)^3}<1.$$
\end{thm}

\begin{proof}
Equation \eqref{eq:recursion} is equivalent to
\begin{eqnarray}\label{eq:recursion-1}
\left(\lambda_{k+1}-\left(1+\frac{2B}{n}\right)\Lambda_k\right)^2\leq\left(1+\frac{2B}{n}\right)^2\Lambda_k^2-
\left(1+\frac{4B}{n}\right)T_k.
\end{eqnarray}
Set $p_{k+1}=\Lambda_{k+1}-\left(1+\frac{2B}{n}\frac{1}{1+k}\right)\Lambda_k$, the equation \eqref{eq:recursion-1} becomes
$$(k+1)^2p_{k+1}^2\leq\left(1+\frac{2B}{n}\right)^2\Lambda_k^2-\left(1+\frac{4B}{n}\right)T_k.$$
By the definition of $F_k$, we have
\begin{eqnarray}\label{eq:recursion-3}
0\leq-(k+1)^2p_{k+1}^2-\frac{2B}{n}\left(1+\frac{2B}{n}\right)\Lambda_k^2+\left(1+\frac{4B}{n}\right)F_k.
\end{eqnarray}
And
\begin{eqnarray}\label{eq:fk}
F_{k+1}&=&\left(1+\frac{2B}{n}\right)\Lambda_{k+1}^2-\frac{k}{k+1}\left(1+\frac{2B}{n}\right)\Lambda_k^2
-\frac{1}{k+1}\lambda_{k+1}^2+\frac{k}{k+1}F_k\\
&=&\left(1+\frac{2B}{n}\right)\left[p_{k+1}+\left(1+\frac{2B}{n}\frac{1}{k+1}\right)\Lambda_k\right]^2-\frac{k}{k+1}\left(1+\frac{2B}{n}\right)\Lambda_k^2\nonumber\\
&&-(k+1)\left[p_{k+1}+\frac{1}{k+1}\left(1+\frac{2B}{n}\right)\Lambda_k\right]^2+\frac{k}{k+1}F_k\nonumber\\
&=&-\left(k-\frac{2B}{n}\right)p_{k+1}^2+\frac{4B}{n}\left(1+\frac{2B}{n}\right)\frac{1}{1+k}p_{k+1}\Lambda_k\nonumber\\
&&+\left(1+\frac{2B}{n}\right)\left[\frac{4B^2}{n^2(1+k)^2}+\frac{2B}{n}\frac{1}{1+k}\right]\Lambda_k^2+\frac{k}{1+k}F_k.\nonumber
\end{eqnarray}
Multiplying \eqref{eq:recursion-3} by $\left[\frac{1}{k+1}+\frac{2B}{n}\left(\frac{1}{(k+1)^2}+\frac{\beta\left(1+\frac{2B}{n}\right)}{(k+1)^3}\right)\right]$ and then adding it to \eqref{eq:fk}, we have
\begin{eqnarray*}
F_{k+1}&\leq&\left(1+\frac{4B}{n}\frac{1}{k+1}+\frac{2B}{n}\frac{\left(1+\frac{4B}{n}\right)}{(k+1)^2}+\frac{2B\beta}{n}\frac{\left(1+\frac{2B}{n}\right)\left(1+\frac{4B}{n}\right)}{(k+1)^3}\right)F_k\\
&&-(2k+1+\frac{2B}{n}\frac{\left(1+\frac{2B}{n}\right)\beta}{k+1})p_{k+1}^2+\frac{4B}{n}\frac{\left(1+\frac{2B}{n}\right)}{(k+1)}p_{k+1}\Lambda_k-
\frac{4\beta B^2}{n^2}\frac{\left(1+\frac{2B}{n}\right)^2}{(k+1)^3}\Lambda_k^2\\
&\leq&\left(1+\frac{4B}{n}\frac{1}{k+1}+\frac{2B}{n}\frac{\left(1+\frac{4B}{n}\right)}{(k+1)^2}+\frac{2B\beta}{n}\frac{\left(1+\frac{2B}{n}\right)\left(1+\frac{4B}{n}\right)}{(k+1)^3}\right)F_k\\
&&-\frac{4\beta B^2}{n^2}\frac{\left(1+\frac{2B}{n}\right)^2}{(k+1)^3}\Lambda_k^2+\frac{4B^2}{n^2}\frac{\left(1+\frac{2B}{n}\right)^2}{(k+1)^2(2k+1)}\Lambda_k^2\\
&&-(2k+1)\left(p_{k+1}-\frac{2B}{n}\frac{\left(1+\frac{2B}{n}\right)}{(k+1)(2k+1)}\Lambda_k\right)^2.
\end{eqnarray*}
Letting $\beta=\frac{k+1}{2k+1}$, we have
\begin{eqnarray}
F_{k+1}&\leq&\left(1+\frac{4B}{n}\frac{1}{k+1}+\frac{2B}{n}\frac{\left(1+\frac{4B}{n}\right)}{(k+1)^2}+\frac{2B}{n}\frac{\left(1+\frac{2B}{n}\right)\left(1+\frac{4B}{n}\right)}{(k+1)^2(2k+1)}\right)F_k.
\end{eqnarray}
Since
\begin{eqnarray}
\left(\frac{k+1}{k}\right)^{\frac{4B}{n}}&=&\left(1-\frac{1}{k+1}\right)^{-\frac{4B}{n}}\\
&=&1+\frac{4B}{n}\frac{1}{k+1}+\frac{1}{2}\frac{4B}{n}\frac{\left(\frac{4B}{n}+1\right)}{\left({k+1}\right)^2}
+\frac{1}{6}\frac{4B}{n}\frac{\left(\frac{4B}{n}+1\right)\left(\frac{4B}{n}+2\right)}{\left({k+1}\right)^3}\nonumber\\
&&+\frac{1}{24}\frac{4B}{n}\frac{\left(\frac{4B}{n}+1\right)\left(\frac{4B}{n}+2\right)\left(\frac{4B}{n}+3\right)}{\left({k+1}\right)^4}+\cdots\nonumber\\
&\geq&1+\frac{4B}{n}\frac{1}{k+1}+\frac{1}{2}\frac{4B}{n}\frac{\left(\frac{4B}{n}+1\right)}{\left({k+1}\right)^2}
+\frac{1}{3}\frac{4B}{n}\frac{\left(\frac{4B}{n}+1\right)\left(\frac{2B}{n}+1\right)}{\left({k+1}\right)^3}\nonumber\\
&&+\frac{1}{4}\frac{4B}{n}\frac{\left(\frac{4B}{n}+1\right)\left(\frac{2B}{n}+1\right)}{\left({k+1}\right)^4},\nonumber
\end{eqnarray}
we have
\begin{eqnarray*}
F_{k+1}&\leq&\left[\left(\frac{k+1}{k}\right)^{\frac{4B}{n}}-\frac{k-1}{3(2k+1)}\frac{2B}{n}\frac{\left(1+\frac{2B}{n}\right)\left(1+\frac{4B}{n}\right)}{(k+1)^3}
-\frac{B}{n}\frac{\left(1+\frac{4B}{n}\right)\left(1+\frac{2B}{n}\right)}{(k+1)^4}\right]F_k\\
&\leq&C(n,k,B)\left(\frac{k+1}{k}\right)^{\frac{4}{n}}F_k,
\end{eqnarray*}
where $C(n,k,B)=1-\frac{B}{3n}\left(\frac{k}{k+1}\right)^{\frac{4B}{n}}\frac{\left(1+\frac{2B}{n}\right)
\left(1+\frac{4B}{n}\right)}{(k+1)^3}<1.$
\end{proof}


Now we are ready to prove Corollary~\ref{coro:ratio}.
\begin{proof}[Proof of Corollary~\ref{coro:ratio}]
Set $B=\frac{1}{1-\lambda_k}.$ It suffices to show that
$$\lambda_{k+1}\leq\left(1+\frac{4B}{n}\right)k^{\frac{2B}{n}}\lambda_1.$$ By Yang type first inequality, \eqref{eq:Yang}, we have \eqref{eq:recursion}. By Theorem~\ref{recursion}, we have
$$F_k\leq C(n,k-1,B)\left(\frac{k}{k-1}\right)^{\frac{4B}{n}}F_{k-1}\leq k^{\frac{4B}{n}}F_1=\frac{2B}{n}k^{\frac{4B}{n}}\lambda_1^2.$$
By \eqref{eq:recursion-3} and noting that $(k+1)p_{k+1}=\lambda_{k+1}-\left(1+\frac{2B}{n}\right)\Lambda_k$, we have
$$\frac{\frac{2B}{n}}{1+\frac{4B}{n}}\lambda_{k+1}^2+\frac{1+\frac{2B}{n}}{1+\frac{4B}{n}}\left(\lambda_{k+1}-
\left(1+\frac{4B}{n}\right)\Lambda_k\right)^2\leq\left(1+\frac{4B}{n}\right)F_k.$$
Hence, we have
$$\lambda_{k+1}^2\leq\frac{n}{2B}\left(1+\frac{4B}{n}\right)^2F_k\leq\left(1+\frac{4B}{n}\right)^2k^{\frac{4B}{n}}\lambda_1^2.$$
\end{proof}

In the following, we adopt the idea in \cite{Ash99} to prove different versions of Yang type second inequality, Hile-Protter type inequality and Payne-Polya-Weinberger type inequality.

For any sequence $\{\lambda_i\}_{i=1}^k$ of nonnegative real numbers, we define \begin{equation}\label{def:nu}\mu_{i}=\frac{1-\lambda_i+\frac2n}{\sum_{j=1}^k(1-\lambda_j+\frac2n)},\ \forall 1\leq i\leq k.\end{equation} If $$\lambda_i\leq 1+\frac2n, \mathrm{\ for\ any\ } 1\leq i\leq k,$$ then $\{\mu_i\}_{i=1}^k$ is a probability measure supported on $k$ points $\{i\}_{i=1}^k.$ We prove another analogue to Yang's second inequality in the discrete setting following the argument of \cite{Ash99}.
\begin{thm}Let $\Omega$ be a finite subset of $\Z^n$ and $\lambda_k$
be the $k$-th eigenvalue of the Dirichlet Laplacian on $\Omega.$ Then
\begin{eqnarray}\label{eq:YangTwo}&&\lambda_{k+1}\leq\frac{1}{\sum_i(1-\lambda_i)}\times\nonumber\\&& \left[\sum_{i}\lambda_i(1-\lambda_i+\frac2n)+\left(\left(\sum_{i}\lambda_i(1-\lambda_i+\frac2n)\right)^2-\sum_i(1-\lambda_i)\sum_{i}\lambda_i^2(1-\lambda_i+\frac4n)\right)^\frac12\right].\end{eqnarray} In particular, if $\lambda_k\leq 1+\frac2n,$ then
\begin{equation*}\label{eq:YangTwo2}\lambda_{k+1}\leq A \sum_{i}\lambda_i\mu_i,
\end{equation*} where \begin{equation}\label{eq:app:2}A:=\left(1+\frac{2k}{n\sum_{i}(1-\lambda_i)}\right)\left(1+\sqrt{1-\left(1+\frac{2k}{n\sum_{i}(1-\lambda_i)}\right)^{-1}}\right),\end{equation} and $\mu_i$ is defined in \eqref{def:nu}.
\end{thm}

\begin{rem} If $\lambda_k\leq 1-\delta$ for some $\delta>0,$ then
$$A\leq C(n,\delta):=(1+\frac{2}{n\delta})\left(1+\sqrt{1-(1+\frac{2}{n\delta})^{-1}}\right).$$
\end{rem}

\begin{proof} For the first assertion, by Yang type first inequality, \eqref{eq:Yang},
$$f(\lambda_{k+1}):=\sum_i(1-\lambda_i) \lambda_{k+1}^2-\lambda_{k+1}\left[\sum_{i}\lambda_i(2+\frac4n-2\lambda_i)\right]+\sum_i\lambda_i^2(1+\frac4n-\lambda_i)\leq 0.$$ As a quadratic inequality in $\lambda_{k+1},$ one obtains that $\lambda_{k+1}$ is less than or equal to the larger root of $f(x),$ which yields \eqref{eq:YangTwo}.

For the second assertion, we estimate the last term in the bracket $[\cdot]$ in \eqref{eq:YangTwo} as follows,
$$\sum_{i}\lambda_i^2(1-\lambda_i+\frac4n)\geq \sum_{i}\lambda_i^2(1-\lambda_i+\frac2n),$$ and obtain
\begin{equation}\label{eq:app:1}\lambda_{k+1}\leq\frac{\frac{2k}{n}+\sum_{i}(1-\lambda_i)}{\sum_i(1-\lambda_i)} \left[\sum_{i}\lambda_i\mu_i+\left((\sum_{i}\lambda_i\mu_i)^2-\frac{\sum_i(1-\lambda_i)}{\frac{2k}{n}+\sum_{i}(1-\lambda_i)}\sum_{i}\lambda_i^2\mu_i\right)^\frac12\right].\end{equation} For $\lambda_k\leq 1+\frac{2}{n},$ $\{\mu_i\}_{i=1}^k$ is a probability measure, which implies that
$$\sum_{i}\lambda_i^2\mu_i\geq (\sum_{i}\lambda_i\mu_i)^2.$$ Plugging it into \eqref{eq:app:1}, we prove the second assertion.


\end{proof}

This result yields other versions of Hile-Protter inequality and Payne-Polya-Weinberger inequality. We omit the proofs here since they are similar to those in Theorem~\ref{thm:HP} and Theorem~\ref{thm:PPW}.
\begin{thm}\label{thm:HP2} Let $\Omega$ be a finite subset of $\Z^n$ and $\lambda_i$
be the $i$-th eigenvalue of the Dirichlet Laplacian on $\Omega.$ Suppose that $\lambda_k\leq 1+\frac2n,$ then
\begin{equation*}\label{eq:HP2}\sum_{i=1}^k\frac{\lambda_i}{\lambda_{k+1}-\lambda_i}\mu_i\geq \frac{1}{A-1},\end{equation*} where $A$ is defined in \eqref{eq:app:2} and $\mu_i$ is defined in \eqref{def:nu}.\end{thm}


\begin{thm}\label{thm:PPW2} Let $\Omega$ be a finite subset of $\Z^n$ and $\lambda_i$
be the $i$-th eigenvalue of the Dirichlet Laplacian on $\Omega.$ Suppose that $\lambda_k\leq 1+\frac2n,$ then
\begin{equation*}\lambda_{k+1}-\lambda_k\leq (A-1)\sum_{i=1}^k{\lambda_i}\mu_i,\end{equation*} where $A$ is defined in \eqref{eq:app:2} and $\mu_i$ is defined in \eqref{def:nu}.\end{thm}


{\bf Acknowledgements.}
B. H. is supported by NSFC, grant no. 11401106. Y. L. is supported by NSFC, grant no. 11671401. Y. S. is supported by NSF of Fujian Province through Grants 2017J01556, 2016J01013.



\bibliography{Dirichleteigenvalue}

\begin{thebibliography}{ESHd09}

\bibitem[AB96]{AshBen96}
M.S. Ashbaugh and R.D. Benguria.
\newblock Bounds for ratios of the first, second, and third membrane
  eigenvalues.
\newblock In {\em Nonlinear problems in applied mathematics}, pages 30--42.
  SIAM, Philadelphia, PA, 1996.

\bibitem[Ash]{Ash99}
M.S. Ashbaugh.
\newblock Isoperimetric and universal inequalities for eigenvalues.
\newblock In {\em Spectral theory and geometry ({E}dinburgh, 1998)}, London
  Math. Soc. Lecture Note Ser., 273.

\bibitem[Ash02]{Ash02}
M.S. Ashbaugh.
\newblock The universal eigenvalue bounds of {P}ayne-{P}\'olya-{W}einberger,
  {H}ile-{P}rotter, and {H}. {C}. {Y}ang.
\newblock {\em Proc. Indian Acad. Sci. Math. Sci.}, 112(1):3--30, 2002.
\newblock Spectral and inverse spectral theory (Goa, 2000).

\bibitem[BHJ14]{BHJ14}
F.~Bauer, B.~Hua, and J.~Jost.
\newblock The dual {C}heeger constant and spectra of infinite graphs.
\newblock {\em Adv. Math.}, 251:147--194, 2014.

\bibitem[CC08]{ChenCheng08}
D.~Chen and Q.-M. Cheng.
\newblock Extrinsic estimates for eigenvalues of the {L}aplace operator.
\newblock {\em J. Math. Soc. Japan}, 60(2):325--339, 2008.

\bibitem[CG98]{CouGri98}
T.~Coulhon and A.~Grigor'yan.
\newblock {Random walks on graphs with regular volume growth}.
\newblock {\em Geom. Funct. Anal.}, 8(4):656--701, 1998.

\bibitem[CH53]{CouHil53}
R.~Courant and D.~Hilbert.
\newblock {\em Methods of mathematical physics. {V}ol. {I}}.
\newblock Interscience Publishers, Inc., New York, N.Y., 1953.

\bibitem[Cha84]{Chavel84}
I.~Chavel.
\newblock {\em Eigenvalues in {R}iemannian geometry}, volume 115 of {\em Pure
  and Applied Mathematics}.
\newblock Academic Press, Inc., Orlando, FL, 1984.
\newblock Including a chapter by Burton Randol, With an appendix by Jozef
  Dodziuk.

\bibitem[CO00]{ChuOd00}
F.~R.~K. Chung and Kevin Oden.
\newblock Weighted graph {L}aplacians and isoperimetric inequalities.
\newblock {\em Pacific J. Math.}, 192(2):257--273, 2000.

\bibitem[CY00]{ChungYau00}
F.R.K. Chung and S.-T. Yau.
\newblock {A Harnack inequality for Dirichlet eigenvalues}.
\newblock {\em J. Graph Theory}, 34(4):247--257, 2000.

\bibitem[CY05]{ChengYang05}
Q.-M. Cheng and H.C. Yang.
\newblock Estimates on eigenvalues of {L}aplacian.
\newblock {\em Math. Ann.}, 331(2):445--460, 2005.

\bibitem[CY06]{ChengYang06}
Q.-M. Cheng and H.C. Yang.
\newblock Inequalities for eigenvalues of {L}aplacian on domains and compact
  complex hypersurfaces in complex projective spaces.
\newblock {\em J. Math. Soc. Japan}, 58(2):545--561, 2006.

\bibitem[CY07]{CY07}
Q.-M. Cheng and H.C. Yang.
\newblock Bounds on eigenvalues of {D}irichlet {L}aplacian.
\newblock {\em Math. Ann.}, 337(1):159--175, 2007.

\bibitem[CY09]{ChengYang09}
Q.-M. Cheng and H.C. Yang.
\newblock Estimates for eigenvalues on {R}iemannian manifolds.
\newblock {\em J. Differential Equations}, 247(8):2270--2281, 2009.

\bibitem[CZL12]{CZL12}
D.~Chen, T.~Zheng, and M.~Lu.
\newblock Eigenvalue estimates on domains in complete noncompact {R}iemannian
  manifolds.
\newblock {\em Pacific J. Math.}, 255(1):41--54, 2012.

\bibitem[CZY16]{CZY16}
D.~Chen, T.~Zheng, and H.C. Yang.
\newblock Estimates of the gaps between consecutive eigenvalues of {L}aplacian.
\newblock {\em Pacific J. Math.}, 282(2):293--311, 2016.

\bibitem[Dod84]{Dod84}
J.~Dodziuk.
\newblock Difference equations, isoperimetric inequality and transience of
  certain random walks.
\newblock {\em Trans. Amer. Math. Soc.}, 284(2):787--794, 1984.

\bibitem[ESHd09]{EHI09}
A.~El~Soufi, E.M. Harrell, II, and I.~Sa\"\i d.
\newblock Universal inequalities for the eigenvalues of {L}aplace and
  {S}chr\"odinger operators on submanifolds.
\newblock {\em Trans. Amer. Math. Soc.}, 361(5):2337--2350, 2009.

\bibitem[Fri93]{Frie93}
J.~Friedman.
\newblock Some geometric aspects of graphs and their eigenfunctions.
\newblock {\em Duke Math. J.}, 69(3):487--525, 1993.

\bibitem[Gri09]{Grigoryan09}
A.~Grigor'yan.
\newblock {\em Analysis on Graphs,
  https://www.math.uni-bielefeld.de/$\sim$grigor/aglect.pdf}.
\newblock 2009.

\bibitem[Har93]{Har93}
E.M. Harrell, II.
\newblock Some geometric bounds on eigenvalue gaps.
\newblock {\em Comm. Partial Differential Equations}, 18(1-2):179--198, 1993.

\bibitem[Har07]{Har07}
E.M. Harrell, II.
\newblock Commutators, eigenvalue gaps, and mean curvature in the theory of
  {S}chr\"odinger operators.
\newblock {\em Comm. Partial Differential Equations}, 32(1-3):401--413, 2007.

\bibitem[HM94]{HarMich94}
E.M. Harrell, II and P.L. Michel.
\newblock Commutator bounds for eigenvalues, with applications to spectral
  geometry.
\newblock {\em Comm. Partial Differential Equations}, 19(11-12):2037--2055,
  1994.

\bibitem[HP80]{HP80}
G.N. Hile and M.H. Protter.
\newblock Inequalities for eigenvalues of the {L}aplacian.
\newblock {\em Indiana Univ. Math. J.}, 29(4):523--538, 1980.

\bibitem[HS97]{HarStu97}
E.M. Harrell, II and J.~Stubbe.
\newblock On trace identities and universal eigenvalue estimates for some
  partial differential operators.
\newblock {\em Trans. Amer. Math. Soc.}, 349(5):1797--1809, 1997.

\bibitem[Leu91]{Leung91}
P.F. Leung.
\newblock On the consecutive eigenvalues of the {L}aplacian of a compact
  minimal submanifold in a sphere.
\newblock {\em J. Austral. Math. Soc. Ser. A}, 50(3):409--416, 1991.

\bibitem[Li80]{Li80}
P.~Li.
\newblock Eigenvalue estimates on homogeneous manifolds.
\newblock {\em Comment. Math. Helv.}, 55(3):347--363, 1980.

\bibitem[LY83]{LY83}
P.~Li and S.T. Yau.
\newblock On the {S}chr\"{o}dinger equation and the eigenvalue problem.
\newblock {\em Comm. Math. Phys.}, 88(3):309--318, 1983.

\bibitem[Pol61]{Pol61}
G.~Polya.
\newblock On the eigenvalues of vibrating membranes.
\newblock {\em Proc. London Math. Soc.}, 11(3):419--433, 1961.

\bibitem[PPW56]{PPW}
L.E. Payne, G.~Polya, and H.F. Weinberger.
\newblock On the ratio of consecutive eigenvalues.
\newblock {\em J. Math. Phys.}, 35:289--298, 1956.

\bibitem[SCY08]{SCY08}
H.~Sun, Q.-M. Cheng, and H.C. Yang.
\newblock Lower order eigenvalues of {D}irichlet {L}aplacian.
\newblock {\em Manuscripta Math.}, 125(2):139--156, 2008.

\bibitem[SY94]{SY94}
R.~Schoen and S.-T. Yau.
\newblock {\em Lectures on differential geometry}.
\newblock Conference Proceedings and Lecture Notes in Geometry and Topology, I.
  International Press, Cambridge, MA, 1994.
\newblock Lecture notes prepared by W.Y. Ding, K.C. Chang [G.Q. Zhang], J.Q.
  Zhong and Y.C. Xu, Translated from the Chinese by Ding and S.Y. Cheng,
  Preface translated from the Chinese by K. Tso.

\bibitem[Tho69]{Thomp69}
C.J. Thompson.
\newblock On the ratio of consecutive eigenvalues in {$N$}-dimensions.
\newblock {\em Studies in Appl. Math.}, 48:281--283, 1969.

\bibitem[Wey12]{Weyl}
H.~Weyl.
\newblock Das asymptotische verteilungsgesetz der eigenwerte linearer
  partieller differentialgleichungen (mit einer anwendung auf die theorie der
  hohlraumstrahlung).
\newblock {\em Math. Ann.}, 71(4):441--479, 1912.

\bibitem[Yan91]{Yang91}
H.C. Yang.
\newblock An estimate of the difference between consecutive eigenvalues.
\newblock {\em Preprintn IC/91/60 of ICTP, Trieste}, 1991.

\bibitem[YY80]{YangYau80}
P.C. Yang and S.-T. Yau.
\newblock Eigenvalues of the {L}aplacian of compact {R}iemann surfaces and
  minimal submanifolds.
\newblock {\em Ann. Scuola Norm. Sup. Pisa Cl. Sci. (4)}, 7(1):55--63, 1980.

\end{thebibliography}
\bibliographystyle{alpha}

\end{document}